\documentclass[twoside,12pt]{article}

\usepackage{times,amsmath,amssymb,amsthm,cite}
\usepackage{mathrsfs,amsmath,mathtools}

\newcommand{\subj}[1]{\par\noindent{\bf AMS Subject Classifications: }#1.}
\newcommand{\keyw}[1]{\par\noindent{\bf Keywords: }#1.}

\numberwithin{equation}{section}
\numberwithin{figure}{section}

\newtheorem{theorem}{Theorem}[section]
\newtheorem{lemma}[theorem]{Lemma}

\newtheorem{proposition}[theorem]{Proposition}
\theoremstyle{definition}
\newtheorem{definition}[theorem]{Definition}

\theoremstyle{remark}

\everymath{\displaystyle}
\date{}
\newcommand{\adsa}
{\vspace{-1in}\normalsize\flushleft
This is a preprint of a paper whose final and definite form will be published in:\\ 
Adv. Dyn. Syst. Appl. ({\tt http://campus.mst.edu/adsa}).\\ 
Submitted July 31, 2012; revised and accepted August 11, 2012.\\\vspace{1mm}\hrule\vspace{5mm}}

\def\di{\displaystyle}

\newcommand{\N}{\mathbb{N}}
\newcommand{\R}{\mathbb{R}}
\newcommand{\LL}{\mathcal{L}}
\newcommand{\CC}{\mathscr{C}}
\newcommand{\PP}{\mathscr{P}}

\newcommand{\W}{\textrm W}
\renewcommand{\L}{\textrm L}
\newcommand{\E}{\textrm E}
\newcommand{\Il}{I^{\alpha}_{-}}
\newcommand{\Ir}{I^{\alpha}_{+}}
\newcommand{\IlC}{I^{1-\alpha}_{-}}
\newcommand{\IrC}{I^{1-\alpha}_{+}}
\newcommand{\Dl}{{_{c}}D^{\alpha}_{-}}
\newcommand{\Dr}{{_{c}}D^{\alpha}_{+}}

\newcommand*{\hooktwoheadrightarrow}{\lhook\joinrel\twoheadrightarrow}
\newcommand{\fonction}[5]{\begin{array}[t]{lrcl}#1 :&#2 &\longrightarrow &#3\\&#4& \longmapsto &#5 \end{array}}

\newcommand{\fonctionsansdef}[3]{\begin{array}[t]{lrcl}#1 :&#2 &\longrightarrow &#3 \end{array}}

\begin{document}

\title{\adsa \center\Large\bf Existence of Minimizers
for Fractional Variational Problems Containing Caputo Derivatives}

\author{{\bf Lo\"ic Bourdin}$^{1}$\\
{\tt bourdin.l@etud.univ-pau.fr}
\and
{\bf Tatiana Odzijewicz}$^{2}$\\
{\tt tatianao@ua.pt}
\and
{\bf Delfim F. M. Torres}$^{2}$\\
{\tt delfim@ua.pt}}

% ---------------------------------

\date{$^1$Laboratoire de Math\'{e}matiques et de leurs Applications\\
Universit\'{e} de Pau et des Pays de l'Adour, Pau, France\\[0.3cm]
$^2$Center for Research and Development in Mathematics and Applications\\
Department of Mathematics, University of Aveiro, Aveiro, Portugal}

\maketitle

% ---------------------------------

\thispagestyle{empty}

% ---------------------------------

\begin{abstract}
We study dynamic minimization problems of the calculus of variations
with Lagrangian functionals containing Riemann--Liouville fractional integrals,
classical and Caputo fractional derivatives.
Under assumptions of regularity, coercivity and convexity,
we prove existence of solutions.
\end{abstract}

% ---------------------------------

\subj{26A33; 49J05}

\keyw{fractional calculus; calculus of variations; existence of minimizers}

% ---------------------------------

\section{Introduction}

For the origin of the calculus of variations with fractional operators
we should look back to 1996-97, when Riewe used non-integer order derivatives
to better describe nonconservative systems in mechanics \cite{MR1401316,MR1438729}.
Since then, numerous works on the fractional variational calculus have been written.
In particular, we can find a comprehensive literature regarding
necessary optimality conditions and Noether's theorem (see, e.g.,
\cite{MR2736825,MR2905870,MR2314503,MR2338631,LoicNoether,MR2550821,MR2454424}).
For the state of the art on the fractional calculus of variations
and respective fractional Euler--Lagrange equations, we refer the reader
to the recent book \cite{MyID:208}. Here we remark that results
addressed to the existence of solutions for problems
of the fractional calculus of variations are rare, being, to the best of our knowledge,
discussed only in \cite{LoicExistence,KlimekExistence}. However, existence
theorems are essential ingredients of the deductive method for solving variational problems,
which starts with the proof of existence, proceeds with application of optimality conditions,
and finishes examining the candidates to arrive to a solution.
These arguments make the question of existence an emergent topic,
which requires serious attention and more interest \cite{MyID:227}.

In this note we discuss the problem of existence of solutions
for fractional variational problems. We consider functionals with Lagrangians
depending on the Riemann--Liouville fractional integral and classical
and Caputo fractional derivatives.
Necessary optimality conditions for such problems
were recently obtained in \cite{MR2861352}.
Here, inspired by the results given in \cite{LoicExistence},
we prove existence of solutions in an appropriate space of functions
and under suitable assumptions of regularity, coercivity and convexity.
For the classical methods of existence of minimizers for variational
functionals we refer the reader to \cite{MR0688142,MR1085948,MR2361288}.

The article is organized as follows. In Section~\ref{sec:basicnotions}
we provide the basic definitions and properties for the fractional operators
used throughout the text. Main notations are fixed. Our results are then formulated
and proved in Section~\ref{sec:results}: in Section~\ref{ssec:tonelli}
we prove existence of minimizers for fractional problems of the calculus
of variations with a Lagrangian containing Caputo derivatives;
Sections~\ref{ssec:regular} and~\ref{ssec:coercive} are devoted
to sufficient conditions implying regularity and coercivity, respectively.
Finally, an example is given in Section~\ref{ex:example}.

% ---------------------------------

\section{Preliminaries}
\label{sec:basicnotions}

We recall here the necessary definitions
and present some properties of the fractional operators under consideration.
Moreover, we fix our notations for later discussions.
The reader interested on fractional analysis is refereed to
the books \cite{MR2218073,MR1658022,MR1347689}.

Let $a,b$ be two real numbers such that $a<b$, let $d\in\N^*$
be the dimension, where $\N^*$ denotes the set of positive integers,
and let $\left\|\cdot\right\|$ denote
the standard Euclidean norm of $\R^d$.
For any $1 \leq r \leq \infty$, we denote
\begin{itemize}
\item by $\L^r := \L^r (a,b;\R^d)$ the usual space of $r$-Lebesgue
integrable functions endowed with its usual norm $\Vert \cdot \Vert_{\L^r}$;
\item by $\W^{1,r} := \W^{1,r} (a,b;\R^d)$ the usual $r$-Sobolev
space endowed with its usual norm $\Vert \cdot \Vert_{\W^{1,r}}$.
\end{itemize}
Furthermore,  $\CC := \CC ([a,b];\R^d)$ will be understood as the standard space
of continuous functions and $\CC^\infty_c := \CC^\infty_c ([a,b];\R^d)$
as the standard space of infinitely differentiable functions compactly supported in $(a,b)$.
Finally, let us remind that the compact embedding
$\W^{1,r} \hooktwoheadrightarrow \CC$ holds for $1 < r \le +\infty$
(see \cite{MR2759829} for a detailed proof).

We define the left and the right Riemann--Liouville fractional integrals
$\Il$ and $\Ir$ of order $\alpha\in\R$, $\alpha >0$, by
\begin{equation*}
\Il[f](t):=\frac{1}{\Gamma(\alpha)}
\int_a^t \frac{f(y)}{(t-y)^{1-\alpha}} dy , \quad t>a
\end{equation*}
and
\begin{equation*}
\Ir[f](t):=\frac{1}{\Gamma(\alpha)}
\int_t^b \frac{f(y)}{(y-t)^{1-\alpha}} dy, \quad t<b,
\end{equation*}
respectively. Here $\Gamma$ denotes the Euler Gamma function.
Note that operators $\Il$ and $\Ir$ are well defined a.e. on $(a,b)$
for any $f\in \L^1$.

Let $0< \alpha <1$ and $\dot{f}$ denote the usual derivative of $f$.
Then the left and the right Caputo fractional derivatives $\Dl$ and $\Dr$
of order $\alpha$ are given by
\begin{equation*}
\Dl[f](t):=\IlC [\dot{f}](t)
~~\textnormal{ and }~~\Dr[f](t):=-\IrC [\dot{f}](t)
\end{equation*}
for all $t\in (a,b]$  and $t \in [a,b)$, respectively.
Note that the Caputo derivatives of a function $f \in W^{1,1}$
are well defined almost everywhere on $(a,b)$.

We make use of the following well-known property
yielding boundedness of Riemann--Liouville fractional integrals in the space $\L^r$.

\begin{proposition}[see, e.g., \cite{MR2218073,MR1347689}]
The left Riemann--Liouville fractional integral
$\Il$ with $\alpha >0$ is a linear and bounded operator in $\L^r$:
\begin{equation*}
\left\|\Il[f]\right\|_{\L^r}
\leq \frac{(b-a)^{\alpha}}{\Gamma(1+\alpha)}\left\|f\right\|_{\L^r}
\end{equation*}
for all $f\in\L^r$, $1 \leq r \leq +\infty$.
\end{proposition}

% ---------------------------------

\section{Main Results}
\label{sec:results}

Along the work $1 < p < \infty$. Let $p'$ denote the adjoint of $p$
and let $\alpha\in\R$, $0< \alpha <1$.
We consider the variational functional
\begin{equation*}
\fonction{\LL}{\E}{\R}{u}{\di \int_a^b L(u,\Il[u],\dot{u},\Dl[u],t) \; dt}
\end{equation*}
and our main goal is to prove existence of minimizers for $\LL$. We assume that $\E$
is a weakly closed subset of $\W^{1,p}$, $\dot{u}$ is the derivative
of $u$ and $\L$ is a Lagrangian of class $\CC^1$:
\begin{equation*}
\fonction{L}{(\R^d)^4 \times [a,b]}{\R}{(x_1,x_2,x_3,x_4,t)}{L(x_1,x_2,x_3,x_4,t).}
\end{equation*}
By $\partial_i L$ we denote the partial derivatives
of $L$ with respect to its $i$th argument.

% ---------------------------------

\subsection{A Tonelli-type Theorem}
\label{ssec:tonelli}

Using general assumptions of regularity, coercivity and convexity,
we prove a fractional analog of the classical Tonelli theorem,
ensuring the existence of a minimizer for $\LL$.

\begin{definition}
We say that $L$ is \emph{regular} if
\begin{itemize}
\item $L(u,\Il[u],\dot{u},\Dl[u],t) \in \L^1$;
\item $\partial_1 L(u,\Il[u],\dot{u},\Dl[u],t) \in \L^1$;
\item $\partial_2 L(u,\Il[u],\dot{u},\Dl[u],t) \in \L^{p'}$;
\item $\partial_3 L(u,\Il[u],\dot{u},\Dl[u],t) \in \L^{p'}$;
\item $\partial_4 L(u,\Il[u],\dot{u},\Dl[u],t) \in \L^{p'}$;
\end{itemize}
for any $u\in\W^{1,p}$.
\end{definition}

\begin{definition}
We say that $\LL$ is \emph{coercive} on $E$ if
\begin{equation*}
\lim\limits_{\substack{\Vert u \Vert_{\W^{1,p}}
\to \infty \\ u \in \E }} \LL (u) = +\infty.
\end{equation*}
\end{definition}

Next result gives a Tonelli-type theorem for Lagrangian functionals
containing fractional derivatives in the sense of Caputo.

\begin{theorem}[Tonelli's existence theorem for fractional variational problems]
\label{thmtonelli}
If
\begin{itemize}
\item $L$ is regular;
\item $\LL$ is coercive on $E$;
\item $L(\cdot,t)$ is convex on $(\R^d)^4$ for any $t \in [a,b]$;
\end{itemize}
then there exists a minimizer for $\LL$.
\end{theorem}

\begin{proof}
Because the Lagrangian $L$ is regular,
$L(u,\Il[u],\dot{u},\Dl[u],t) \in \L^1$
and $\LL (u)$ exists in $\R$.
Let $(u_n)_{n \in \N} \subset \E$
be a minimizing sequence satisfying
\begin{equation}
\label{eq:1}
\LL (u_n) \longrightarrow \inf\limits_{u \in \E} \LL (u) < +\infty.
\end{equation}
Coercivity of $\LL$ implies boundedness of $(u_n)_{n \in \N}$ in $\W^{1,p}$.
Moreover, since $\W^{1,p}$ is a reflexive Banach space, there exists $\bar{u}$
and a subsequence of $(u_n)_{n \in \N}$,
that we still denote as  $(u_n)_{n \in \N}$, such that
$u_n  \xrightharpoonup[]{\W^{1,p}}\bar{u}$.
Furthermore, since $\E$ is a weakly closed subset
of $\W^{1,p}$, $\bar{u} \in \E$.
On the other hand, from the convexity of $L$, we have
\begin{multline}
\label{eq00}
\LL (u_n) \geq \LL (\bar{u}) + \di \int_a^b \partial_1 L \cdot (u_n-\bar{u})
+ \partial_2 L \cdot (\Il[u_n]-\Il[\bar{u}]) \\
+ \partial_3 L \cdot (\dot{u}_n-\dot{\bar{u}})
+ \partial_4 L \cdot (\Dl[u_n]-\Dl[\bar{u}])  \; dt
\end{multline}
for any $n \in \N$, where $\partial_i L$ is taken in
$(\bar{u},\Il[\bar{u}],\dot{\bar{u}},\Dl[\bar{u}],t)$,
$i=1,2,3,4$.
Now, because $L$ is regular, $(u_n)_{n \in \N}$ is weakly convergent to
$\bar{u}$ in $\W^{1,p}$, $\Il$ is linear bounded from $\L^p$ to $\L^p$
and, since the compact embedding $\W^{1,p} \hooktwoheadrightarrow \CC$ holds,
one concludes that
\begin{itemize}
\item $\partial_3 L(\bar{u},\Il[\bar{u}],\dot{\bar{u}},\Dl[\bar{u}],t)
\in \L^{p'}$ and $\dot{u_n} \xrightharpoonup[]{\L^p} \dot{\bar{u}}$;
\item $\partial_4 L(\bar{u},\Il[\bar{u}],\dot{\bar{u}},\Dl[\bar{u}],t)
\in \L^{p'}$ and $\Dl[u_n] \xrightharpoonup[]{\L^p} \Dl[\bar{u}]$;
\item $\partial_1 L(\bar{u},\Il[\bar{u}],\dot{\bar{u}},\Dl[\bar{u}],t)
\in \L^1$ and $u_n \xrightarrow[]{\L^\infty} \bar{u}$;
\item $\partial_2 L(\bar{u},\Il[\bar{u}],\dot{\bar{u}},\Dl[\bar{u}],t)
\in \L^{p'}$ and $\Il[u_n] \xrightarrow[]{\L^p} \Il[\bar{u}]$.
\end{itemize}
Finally, returning to \eqref{eq:1}
and taking $n\rightarrow\infty$ in inequality \eqref{eq00}, we obtain that
\begin{equation*}
\inf\limits_{u \in \E} \LL (u) \geq \LL (\bar{u}) \in \R,
\end{equation*}
which completes the proof.
\end{proof}

In order to make the hypotheses of our Theorem~\ref{thmtonelli} more concrete,
in Sections~\ref{ssec:regular} and \ref{ssec:coercive}
we prove more precise sufficient conditions on the Lagrangian $L$,
that imply regularity and coercivity of functional $\LL$.
For this purpose we define a family of sets $\PP_M$ for any $M\geq 1$.

% ---------------------------------

\subsection{Sufficient Condition for a Lagrangian $\mathbf{L}$ to be Regular}
\label{ssec:regular}

For $M\geq 1$, we define $\PP_M$ to be the set of maps
$P:(\R^d)^4\times [a,b]\rightarrow \R^{+}$ such that
for any $(x_1,x_2,x_3,x_4,t) \in (\R^d)^4 \times [a,b]$
\begin{equation*}
P(x_1,x_2,x_3,x_4,t) = \di \sum_{k=0}^{N} c_k(x_1,t)
\Vert x_2 \Vert^{d_{2,k}} \Vert x_3 \Vert^{d_{3,k}} \Vert x_4 \Vert^{d_{4,k}}
\end{equation*}
with $ N \in \N$ and where, for any $k=0,\ldots,N$,
$\fonctionsansdef{c_k}{\R^d \times [a,b]}{\R^+}$
is continuous and satisfies $d_{2,k}+ d_{3,k} + d_{4,k} \leq p/M$.

The following lemma holds for the family of maps $\PP_M$.

\begin{lemma}
\label{lemregP}
Let $M \geq 1$ and $P \in \PP_M$. Then, for any $u \in \W^{1,p}$, we have
\begin{equation*}
\; P(u,\Il[u],\dot{u},\Dl[u],t) \in \L^{M}.
\end{equation*}
\end{lemma}

\begin{proof}
Because $c_k(u,t)$ is continuous for any $k=0,\ldots,N$, it is in $\L^\infty$.
We also have $\Vert \Il[u] \Vert^{d_{2,k}} \in \L^{p/d_{2,k}}$,
$\Vert \dot{u} \Vert^{d_{3,k}} \in \L^{p/d_{3,k}}$
and $\Vert \Dl[u] \Vert^{d_{4,k}} \in \L^{p/d_{4,k}}$. Consequently,
\begin{equation*}
c_k(u,t) \Vert \Il[u] \Vert^{d_{2,k}}
\Vert \dot{u} \Vert^{d_{3,k}} \Vert  \Dl[u] \Vert^{d_{4,k}} \in \L^r
\end{equation*}
with $r=p/(d_{2,k} + d_{3,k} +d_{4,k}) \geq M$. The proof is complete.
\end{proof}

With the help of Lemma~\ref{lemregP},
it is easy to prove the following
sufficient condition on the Lagrangian $L$,
which implies its regularity.

\begin{proposition}
\label{prop1}
If there exists $P_0 \in \PP_{1}$, $P_1 \in \PP_{1}$,
$P_2 \in \PP_{p'}$, $P_3 \in \PP_{p'}$ and $P_4 \in \PP_{p'}$
such that
\begin{itemize}
\item $ \vert L(x_1,x_2,x_3,x_4,t) \vert \leq P_0(x_1,x_2,x_3,x_4,t) $;
\item $ \Vert \partial_1 L(x_1,x_2,x_3,x_4,t) \Vert \leq P_1(x_1,x_2,x_3,x_4,t) $;
\item $ \Vert \partial_2 L(x_1,x_2,x_3,x_4,t) \Vert \leq P_2(x_1,x_2,x_3,x_4,t) $;
\item $ \Vert \partial_3 L(x_1,x_2,x_3,x_4,t) \Vert \leq P_3(x_1,x_2,x_3,x_4,t) $;
\item $ \Vert \partial_4 L(x_1,x_2,x_3,x_4,t) \Vert \leq P_4(x_1,x_2,x_3,x_4,t) $;
\end{itemize}
for any $(x_1,x_2,x_3,x_4,t) \in (\R^d)^4 \times [a,b]$, then $L$ is regular.
\end{proposition}

The coercivity assumption in Theorem~\ref{thmtonelli} is strongly dependent
on the set $\E$. In Section~\ref{ssec:coercive} we provide an example of such set.
Moreover, with such choice for $\E$, we give a sufficient condition
on the Lagrangian $\L$ implying coercivity of $\LL$.

% ---------------------------------

\subsection{Sufficient Condition for a Functional $\mathbf{\LL}$ to be Coercive}
\label{ssec:coercive}

Consider $u_0 \in \R^d$ and $\E = \W^{1,p}_a$, where
$\W^{1,p}_a := \{ u \in \W^{1,p}, \; u(a)=u_0 \}$. We note that
$\W^{1,p}_a$ is a weakly closed subset of $\W^{1,p}$
because of the compact embedding $\W^{1,p} \hooktwoheadrightarrow \CC$.

The following lemma is important in the proof of Proposition~\ref{prop1b}.

\begin{lemma}
\label{lemdom}
There exist $A_0$, $A_1 \geq 0$ such that
\begin{itemize}
\item $\Vert u \Vert_{\L^\infty} \leq A_0 \Vert \dot{u} \Vert_{\L^p} + A_1$;
\item $\Vert \Il[u] \Vert_{\L^p} \leq A_0 \Vert \dot{u} \Vert_{\L^p} + A_1$;
\item $\Vert \Dl[u] \Vert_{\L^p} \leq A_0 \Vert \dot{u} \Vert_{\L^p} + A_1$;
\end{itemize}
for any $u \in \W^{1,p}_a$.
\end{lemma}

\begin{proof}
It is easy to see that boundedness of the left Riemann--Liouville fractional integral $\Il$
implies the last inequality. In the case of the second inequality,
we have $\Vert u \Vert_{\L^p} \leq \Vert u-u_0 \Vert_{\L^p}
+ \Vert u_0 \Vert_{\L^p} \leq (b-a) \Vert \dot{u} \Vert_{\L^p} + (b-a)^{1/p} \Vert u_0 \Vert$
for any $u \in \W^{1,p}_a$. Therefore, using again the boundedness of $\Il$,
we arrive to the desired conclusion. Finally,
let us consider the first inequality. We have
$\Vert u \Vert_{\L^\infty} \leq \Vert u-u_0 \Vert_{\L^\infty}
+ \Vert u_0 \Vert \leq \Vert \dot{u} \Vert_{\L^1}
+ \Vert u_0 \Vert \leq (b-a)^{1/p'} \Vert \dot{u} \Vert_{\L^p} + \Vert u_0 \Vert$
for any $u \in \W^{1,p}_a$. The proof is completed by defining $A_0$ and $A_1$
as the maximum of the appearing constants.
\end{proof}

Next proposition gives a sufficient condition for the coercivity of $\LL$.

\begin{proposition}
\label{prop1b}
Assume that
\begin{equation*}
L(x_1,x_2,x_3,x_4,t) \geq c_0 \Vert x_3 \Vert^{p}
+ \di \sum_{k=1}^{N} c_k \Vert x_1 \Vert^{d_{1,k}}
\Vert x_2 \Vert^{d_{2,k}} \Vert x_3 \Vert^{d_{3,k}} \Vert x_4 \Vert^{d_{4,k}}
\end{equation*}
for any $(x_1,x_2,x_3,x_4,t) \in (\R^d)^4 \times [a,b]$,
where $c_0 > 0$, $c_k \in \R$, $N \in \N^*$, and
\begin{equation}
\label{eq:assump}
0 \leq d_{1,k}+d_{2,k}+d_{3,k}+d_{4,k} < p
\end{equation}
for any $k=1,\ldots,N$.
Then, $\LL$ is coercive on $\W^{1,p}_a$.
\end{proposition}

\begin{proof}
First, let us define $r = p/(d_{2,k}+d_{4,k}+ d_{3,k}) \geq 1$.
Applying H\"older's inequality, one can easily prove that
\begin{equation*}
\LL (u) \geq c_0 \Vert \dot{u} \Vert_{\L^p}^p - (b-a)^{1/r'} \di
\sum_{k=1}^N \vert c_k \vert \Vert u \Vert^{d_{1,k}}_{\L^\infty}
\Vert \Il[u] \Vert^{d_{2,k}}_{\L^p}  \Vert \dot{u} \Vert^{d_{3,k}}_{\L^p}
\Vert \Dl[u] \Vert^{d_{4,k}}_{\L^p}
\end{equation*}
for any $u \in \W^{1,p}_a$.
Moreover, from Lemma~\ref{lemdom} and \eqref{eq:assump}, we obtain that
\begin{equation*}
\lim\limits_{\substack{\Vert \dot{u} \Vert_{\L^{p}} \to \infty\\
u \in \W^{1,p}_a }} \LL (u) = +\infty.
\end{equation*}
Finally, applying again Lemma~\ref{lemdom}, we have that
\begin{equation*}
\Vert \dot{u} \Vert_{\L^{p}} \to \infty
\Longleftrightarrow \Vert u \Vert_{\W^{1,p}} \to \infty
\end{equation*}
in $\W^{1,p}_a$. Therefore, $\LL$ is coercive on $\W^{1,p}_a$.
The proof is complete.
\end{proof}

In the next section we illustrate our results through an example.

% ---------------------------------

\section{An Illustrative Example}
\label{ex:example}

Consider the following fractional problem of the calculus of variations:
\begin{equation}
\label{eq:ex}
\begin{gathered}
\LL(u) = \int_a^b \Vert u \Vert^2 + \Vert \Il[u] \Vert^2
+ \Vert \dot{u} \Vert^2 + \Vert \Dl[u] \Vert^2  \; dt \longrightarrow \min_{u \in \W^{1,2}}\\
u(a) = u_0.
\end{gathered}
\end{equation}
It is not difficult to verify that the Lagrangian $L$ for this problem
is convex and satisfies the hypotheses of Propositions~\ref{prop1} and~\ref{prop1b} with $p=2$.
Therefore, it follows from Theorem~\ref{thmtonelli}
that there exists a solution for problem \eqref{eq:ex}.
Such minimizer can be determined using the optimality conditions proved
in \cite{MR2861317,MyID:208} and approximated by the numerical methods
developed in \cite{MyID:225,MyID:235}.

% ---------------------------------

\section*{Acknowledgements}

The authors are grateful to the support and the good working conditions
of the {\it Center for Research and Development
in Mathematics and Applications}, University of Aveiro.
They would also like to thank an anonymous referee for his/her
careful reading of the manuscript and for suggesting useful changes.

% ---------------------------------

% ---------------------------------

\end{document}